\documentclass[12pt]{amsproc}
\usepackage{amssymb,amsthm,amsmath,anysize,array,lmodern,stmaryrd}
\usepackage[foot]{amsaddr}
\usepackage{booktabs, float, xcolor}
\usepackage{enumerate}
\usepackage{url}

\usepackage{hyperref}
\hypersetup{
    colorlinks,
    linkcolor={red!60!black},
    citecolor={green!60!black},
    urlcolor={blue!60!black}
}
\usepackage{geometry}
\usepackage{bbm}
\usepackage{microtype}
\usepackage[skins,breakable]{tcolorbox}

\usepackage{tikz}
\usetikzlibrary{calc}

\theoremstyle{plain}
\newtheorem{Theorem}{Theorem}[section]
\newtheorem{Corollary}[Theorem]{Corollary}

\newtheorem{Lemma}[Theorem]{Lemma}
\newtheorem{Remark}[Theorem]{Remark}

\theoremstyle{definition}
\newtheorem{Definition}[Theorem]{Definition}

\newtheorem{example}[Theorem]{Example}
\tcolorboxenvironment{example}{
    blanker,
    breakable,
    left=5mm,
    before skip=10pt,
    after skip=10pt,
    borderline west={5pt}{0pt}{gray!65},
    pad before break=2pt,
    pad after break=2pt,
    use color stack,
}

\DeclareMathOperator{\N}{\mathbb{N}}
\DeclareMathOperator{\C}{\mathbb{C}}

\DeclareMathOperator{\F}{\mathbb{F}}
\DeclareMathOperator{\PG}{\mathrm{PG}}

\usetikzlibrary{positioning}

\newcommand{\youngdiagram}[1]{
  \begin{tikzpicture}[scale=0.2,baseline=0pt]
    \foreach \row [count=\i] in {#1} {
      \foreach \x in {1,...,\row} {
        \draw (\x,-\i) rectangle +(1,1);
      }
    }
  \end{tikzpicture}
}

\newcommand\ltriangle{\trianglelefteqslant}
\newcommand\rtriangle{\trianglerighteqslant}

\renewcommand{\le}{\leqslant}
\renewcommand{\ge}{\geqslant}

\usepackage{listings}

\definecolor{dkgreen}{rgb}{0,0.6,0}
\definecolor{gray}{rgb}{0.5,0.5,0.5}
\definecolor{mauve}{rgb}{0.58,0,0.82}

\title[On the association scheme of perfect matchings and their designs]{On the association scheme of perfect matchings\\ and their designs}

\author{John Bamberg}
\address[Bamberg]{Department of Mathematics and Statistics, 
The University of Western Australia, 35 Stirling Highway, Perth, W.A. 6009, Australia.}
\email{john.bamberg@uwa.edu.au}

\author{Lukas Klawuhn}
\address[Klawuhn]{Department of Mathematics, Paderborn University, Warburger Str.\ 100, 33098 Paderborn, Germany.}
\email{klawuhn@math.upb.de}
\date{18 March 2026}

\dedicatory{Dedicated to the memory of Kai-Uwe Schmidt, who started us on the journey of
this research.}

\begin{document}

\begin{abstract}
We investigate generalisations of 1-factorisations and hyperfactorisations of the complete graph $K_{2n}$. We show that they are special subsets of the association scheme obtained from the Gelfand pair $(S_{2n},S_2 \wr S_n)$. This unifies and extends results by Cameron (1976) and gives rise to new existence and non-existence results. Our methods involve working in the group algebra $\mathbb{C}[S_{2n}]$ and using the representation theory of $S_{2n}$.
\end{abstract}

\maketitle

\section{Introduction}

It is well-known that the complete graph $K_{2n}$ on an even number of vertices can always be decomposed into perfect matchings, called a \emph{1-factorisation}. Equivalently, we can view a 1-factorisation as a set of perfect matchings of $K_{2n}$ such that every edge of $K_{2n}$ is contained in exactly one perfect matching. From this point of view, one can generalise 1-factorisations to \emph{hyperfactorisations}, studied by Jungnickel and Vanstone \cite{JunVan1987}. A \emph{hyperfactorisation} is a set of perfect matchings of $K_{2n}$ such that every pair of disjoint edges is contained in exactly $c$ perfect matchings, where $c$ is a constant called the \emph{index} of the hyperfactorisation. Boros, Jungnickel, and Vanstone \cite{BorJunVan1991} showed that $K_{2n}$ has a non-trivial hyperfactorisation for every $n \ge 5$.

Cameron \cite{Cam1976} generalised hyperfactorisations (or `parallelisms of $\binom{X}{t}$') to partition systems. He defines an \emph{$s$-$(t,n)$ partition system} to be a collection $S$ of partitions of an $n$-set $X$ into $t$-subsets having the property that given any $s$ disjoint $t$-subsets of $X$, there is a unique partition in $S$ which has all of the given $t$-subsets as parts. This is a natural way to generalise 1-factorisations and hyperfactorisations of complete graphs. Indeed, a 1-factorisation of $K_n$ is simply a $1$-$(2,n)$ partition system. Independently, Stinson \cite{Stinson} defined \emph{hyperresolutions} of a resolvable block design, which turn out to be $2$-$(t,n)$ partition systems. Cameron constructed infinitely many examples of $2$-$(2,n)$ partition systems from hyperovals of projective planes of even order.

We generalise the notion of $s$-$(2,n)$ partition systems to \emph{$\lambda$-factorisations} by replacing the parameter $s$ with a partition $\lambda$ and use the theory of association schemes to investigate them. In an $s$-$(2,n)$ partition system, any set of $s-1$ subsets of size $2$ appears a constant number of times. The counting involved does not give rise to restrictions on the parameters $s$ and $n$ (see \cite[Theorem 7.2]{Cam1976}). We show that similar counting arguments work for $\lambda$-factorisations and give rise to non-trivial divisibility conditions by comparing different partitions $\lambda$.

This work can be considered a follow-up to the work of Rands \cite{Rands} on the association
scheme of perfect matchings of the complete graph $K_{2n}$. Rands calculated the eigenvalues
of the association scheme for $n\in\{4,5,6\}$, and he also observed that we can view
abstract hyperovals and 1-factorisations as cliques for this association scheme. We take
this viewpoint further and consider $\lambda$-factorisations as designs in the association scheme,
and use Delsarte theory to produce results on the existence of such designs. 
Our main result is a characterisation of $\lambda$-factorisations as designs:

\begin{Theorem}[Paraphrase of Theorem \ref{theorem:main1}]
    A non-empty set of perfect matchings is a $\lambda$-factorisation if and only if its dual degree set contains no partition $\mu\vdash n$ with $\lambda \ltriangle \mu \neq (n)$.
\end{Theorem}

In Section \ref{section:designs}, we give several interesting corollaries of this result.
One of which (Corollary \ref{cor:comparison_t_factorisations}),
is that for every $n\ge 4$, an $(n-2,2)$-factorisation\footnote{This is a set $D$ of perfect matchings of $K_{2n}$ such that for every 4-subset $F$ of the vertices, there are precisely $c$ elements of $D$ which have two edges contained in $F$ (and $c$ is the \emph{index}).} of index coprime to 3 (in particular index 1) can only exist 
if $n \equiv 0 \pmod 3$. Also, in Corollary \ref{corollary:2111} we show that 
there is no $(2,1,\ldots,1)$-factorisation for all $n\ge 4$,
generalising a result of Cameron \cite[Theorem 7.6]{Cam1976} on partition systems.

\section{Preliminaries}

In this section, we fix some notation and introduce the basic definitions of partitions and irreducible characters.

\subsection{Partitions}

Many combinatorial objects can be indexed by partitions. Here, we give a brief overview of the main definitions. We mostly follow \cite[Ch. I]{Mac1995}.

An (integer) \emph{partition} is a sequence $\lambda = (\lambda_1,\lambda_2,\ldots)$ of non-negative integers that sum to a finite number and satisfy $\lambda_i \ge \lambda_{i+1}$ for every $i \in \N$. The \emph{size} of $\lambda$ is the number $|\lambda| = \sum_{i\ge 1} \lambda_i$ and if $|\lambda| = n$, we say that $\lambda$ is a partition of $n$, denoted by $\lambda\vdash n$. We often write partitions as finite sequences $(\lambda_1,\ldots,\lambda_k)$ only writing down the non-zero parts. The number of non-zero parts of $\lambda$ is the \emph{length} of $\lambda$, denoted by $l(\lambda)$. We use the notation $\lambda!$ for the product of the factorials $\lambda_i!$.

For a partition $\lambda \vdash n$, we denote by $2\lambda$ the partition of $2n$ with parts $(2\lambda_1,2\lambda_2,\ldots)$. If $\lambda \vdash n$, then the \emph{Young diagram} of $\lambda$ is an array of $n$ boxes with left-justified rows and top-justified columns such that row $i$ contains exactly $\lambda_i$ boxes. 
We use the dominance order $\ltriangle$ to compare partitions. Let $\lambda,\mu \vdash n$ be two partitions. We write $\lambda \ltriangle \mu$ and say that $\lambda$ is \emph{dominated} by $\mu$ or $\mu$ \emph{dominates} $\lambda$ if
\[
	\sum_{i=1}^k \lambda_i \le \sum_{i=1}^k \mu_i \quad \text{ for all } k \ge 1.
\]
A helpful fact is that we have $\lambda \ltriangle \mu$ if and only if the Young diagram of $\lambda$ can be transformed into the Young diagram of $\mu$ by moving boxes from the end of a row of $\lambda$ to a higher row one at a time. When interpreting a partition as a sequence, this corresponds to adding vectors of the form $(0,\ldots,0,1,0,\ldots,0,-1,0,\ldots)$ to the partition.

We will also make heavy use of \emph{set partitions}.

\begin{Definition}
Let $n \in \mathbb{N}$. A \emph{set partition} is a set $\{P_1,P_2,\ldots\}$ of non-empty disjoint subsets $P_i \subseteq \{1,\ldots,n\}$ such that $P_1 \cup P_2 \cup \ldots = \{1,\ldots,n\}$. The partition $\lambda$ whose parts are the sizes $|P_i|$ (ordered non-increasingly) is called the \emph{shape} of the set partition.
An \emph{ordered set partition} is a tuple $(P_1,P_2,\ldots)$ of non-empty disjoint subsets $P_i \subseteq \{1,\ldots,n\}$ such that $\{P_1,P_2,\ldots\}$ is a set partition. The shape of an ordered set partition is the tuple $(|P_1|,|P_2|,\ldots)$ (which is not necessarily ordered non-increasingly).
\end{Definition}

A \emph{perfect matching} of the complete graph $K_{2n}$ is simply a set partition of the set $\{1,\ldots,2n\}$ into $n$ subsets of size $2$. Its shape is $(2,\ldots,2) \vdash 2n$. It is well-known that $K_{2n}$ has exactly
\[
    (2n-1)!! = (2n-1)\cdot (2n-3) \cdot \cdots  \cdot 5 \cdot 3 \cdot 1
\]
perfect matchings. 

\subsection{Characters}

Since many formulas in association schemes defined on groups involve irreducible characters, we now turn to representation and character theory. We refer the reader to \cite{Ser1977} for a more extensive introduction to the representation theory of finite groups. For a finite group $G$,
we will refer to ordinary representations and characters, and the \emph{degree} of a character will
be the dimension of the vector space for the associated representation.

We remind the reader that the class functions of $G$ form an inner product space under
the inner product given by
\[
	\langle \chi, \psi \rangle = \frac{1}{|G|}\sum_{g \in G} \chi(g)\overline{\psi(g)}
\]
and the irreducible characters form an orthonormal basis of the space of class functions with respect to $\langle \cdot,\cdot \rangle$. 

Furthermore, if $\varphi$ is an arbitrary character and $\chi$ is an irreducible character, then $\langle \varphi,\chi \rangle$ is the \emph{multiplicity} of $\chi$ in the decomposition of $\varphi$ into irreducible characters. If $\langle \varphi,\chi \rangle \neq 0$, then $\chi$ is called an \emph{irreducible constituent} of $\varphi$. If the multiplicity of every irreducible constituent of $\varphi$ is equal to 1, then $\varphi$ is called \emph{multiplicity-free}.

Let $H \le G$. If $\chi$ is a character of $G$, then restricting $\chi$ to $H$ gives a character of $H$. This character is called the \emph{restriction} of $\chi$ to $H$ and denoted by $\chi \downarrow^G_H$. There is a way to extend a character of a subgroup $H$ to a character of $G$ called \emph{induction}. If $\chi$ is a character of $H$, then the induced character is denoted by $\chi \uparrow^G_H$.
The function $1_G$ with $1_G(g) = 1$ for all $g \in G$ is a character of $G$, called the \emph{trivial character} of $G$.

\subsection{Character Theory of the Symmetric Group}

It is well-known that the irreducible characters of the symmetric group $S_n$ are indexed by the partitions of $n$. If $\lambda \vdash n$, then $\chi^\lambda$ denotes the irreducible character of the symmetric group $S_n$ indexed by the partition $\lambda$.
Notice that the symmetric group $S_n$ naturally acts on ordered set partitions. The stabiliser of an ordered set partition is of the form
\[
    S_\lambda = S_{\lambda_1} \times S_{\lambda_2} \times \cdots \times S_{\lambda_{l(\lambda)}}.
\]
The subgroup $S_\lambda$ is called a \emph{Young subgroup} of shape $\lambda$. All Young subgroups of the same shape are conjugate to each other. Hence, all characters obtained by inducing a character from a Young subgroup of a given shape to the whole group are equal.
For two partitions $\lambda,\mu$ of $n$, we have ``Young's rule" (cf. \cite[Theorem 2.11.2]{Sagan2001}): $\langle 1\uparrow_{S_{\lambda}}^{S_{n}},\chi^\mu \rangle \neq 0$ if and only if $\lambda \ltriangle \mu$.

\section{Association Schemes and Gelfand Pairs}

\subsection{Association schemes}

Association schemes are a helpful structure to investigate interesting combinatorial structures and they provide powerful techniques to analyse them. We refer the reader to \cite{BanIto1984} for an extensive introduction to association schemes.
An \emph{association scheme} (with $d$ classes) is a pair $(X,\mathcal{R})$ of a finite set $X$ with $|X| \ge 2$ and a family of relations $\mathcal{R} = (R_i)_{i=0,\ldots,d}$ subject to the following conditions:
\begin{enumerate}
    \item[(A0)]{ $\mathcal{R}$ forms a set partition of $X\times X$; that is, 
        $R_i \cap R_j = \varnothing$ for $i \neq j$ and $\bigcup\limits_{i = 0}^{d} R_i = X \times X$.
    }
    \item[(A1)]{
    $R_0$ is the `equality' relation; that is, 
        $R_0 = \{(x,x) \in X \times X \mid x \in X\}$.
    }
    \item[(A2)]{
    $\mathcal{R}$ is closed under converses; that is,
        $R_i^\top = \{(y,x) \in X \times X \mid (x,y) \in R_i\} = R_j$ for some $j \in \{0,\ldots,d\}$.
    }
    \item[(A3)]{
       Let $i,j,k \in \{0,\ldots,d\}$. If $(x,y) \in R_k$, then the number $p_{ij}^k$ of elements $z \in X$ such that $(x,z) \in R_i$ and $(z,y) \in R_j$ only depends on $i$, $j$ and $k$ but not on the choice of $(x,y)$.
    }
    \item[(A4)]{
        For every $i,j,k \in \{0,\ldots,d\}$ we have $p_{ij}^k = p_{ji}^k$.
    }
\end{enumerate}

For finite and non-empty sets $X$ and $Y$, let $\C(X,Y)$ denote the set of all complex $|X| \times |Y|$-matrices with rows indexed by $X$ and columns indexed by $Y$. For a matrix $A \in \C(X,Y)$ and $x\in X$, $y \in Y$, the $(x,y)$-entry of $A$ is denoted by $A(x,y)$. If $|Y|=1$, then we omit $Y$, so $\C(X)$ is the set of complex column vectors indexed by $X$.

Let $(X,R)$ be an association scheme. For $i \in \{0,\ldots,d\}$, let $A_i \in \C(X,X)$ be the $i$-th
\emph{adjacency matrix} given by
\begin{equation}\label{eqn:def_incidence_matrices_general}
	A_i(x,y) = \begin{cases} 1,& \text{ if } (x,y) \in R_i\\
						       0,& \text{ otherwise}.
				\end{cases}
\end{equation}
Let $\mathbb{A} = \operatorname{sp}(A_0,\ldots,A_d)$ be the vector space generated by $A_0,\ldots,A_d$ over the complex numbers. Then $\mathbb{A}$ is a commutative matrix algebra with basis $A_0,\ldots,A_d$ that contains the identity matrix and is closed under conjugate transposition. Thus, the zero-one-matrices $A_i$ define an association scheme. The algebra $\mathbb{A}$ is called the \emph{Bose--Mesner algebra} of the association scheme.

We can rewrite the definition of an association scheme in terms of matrices. Let $X$ be a finite set with $|X| \ge 2$ and $(A_i)_{i=0,\ldots,d}$ be a family of matrices in $\C(X,X)$ with entries $0$ or $1$ subject to the following conditions:
\begin{enumerate}
    \item[(A0')]{
        $\sum_{i=0}^{d} A_i = J$ where $J$ denotes the all-ones matrix;
    }
    \item[(A1')]{
        $A_0 = I$ where $I$ denotes the identity matrix;
    }
    \item[(A2')]{
        $A_i^\top = A_j$ for some $j \in \{0,\ldots,d\}$;
    }
    \item[(A3')]{
        For every $i,j \in \{0,\ldots,d\}$ we have $A_iA_j = \sum_{k=0}^d p_{ij}^k A_k$;
    }
    \item[(A4')]{
        For every $i,j \in \{0,\ldots,d\}$ we have $A_iA_j = A_jA_i$.
    }
\end{enumerate}
The matrices $A_i$ of an association scheme fulfil these conditions and if a family of matrices fulfils these conditions, it generates the Bose--Mesner algebra of an association scheme.

Since the Bose--Mesner algebra $\mathbb{A}$ of an association scheme is commutative and closed under conjugate transposition, all of its elements are normal and can be simultaneously diagonalised via a unitary matrix. Therefore, there exists a basis $E_0,\ldots,E_d$ of $\mathbb{A}$ consisting of Hermitian matrices with the property
\begin{equation}\label{eqn:minimal_idempotent_property_general}
	E_k E_l = \delta_{kl} E_k.
\end{equation}
The matrices $E_k$ are the minimal idempotents of $\mathbb{A}$.

Let $V_k$ be the column space of $E_k$. The spaces $V_k$ are called the \emph{eigenspaces} of the association scheme because they are the common eigenspaces of the matrices $A_i$. From (\ref{eqn:minimal_idempotent_property_general}), the $V_k$ are pairwise orthogonal and $\C(X) = \bigoplus_{k = 0}^d V_k$.
\par
Let $Y \subseteq X$ be a non-empty subset. We can associate two sequences of numbers with $Y$: the \emph{inner distribution} $a$ and the \emph{dual distribution} $a'$. The inner distribution of $Y$ is the tuple $(a_0,\ldots,a_d)$, where
\begin{equation}\label{eqn:def_inner_distribution_general}
	a_i = \frac{1}{|Y|}\sum_{x,y \in Y} A_i(x,y),
\end{equation}
and the dual distribution of $Y$ is the tuple $(a'_0,\ldots,a'_d)$, where
\begin{equation}\label{eqn:def_dual_distribution_general}
	a'_k = \frac{|X|}{|Y|}\sum_{x,y \in Y} E_k(x,y).
\end{equation}
It is immediate that the inner distribution is non-negative. The same holds for the dual distribution. Let $\mathbbm{1}_Y \in \C(X)$ be the characteristic vector of $Y$, so $\mathbbm{1}_Y(x) = 1$ if $x \in Y$ and $\mathbbm{1}_Y(x) = 0$ otherwise. Since the matrix $E_k$ is Hermitian, it follows that
\begin{equation}\label{eqn:dual_distr_orthogonality_general}
	\frac{|Y|}{|X|} a'_k = \mathbbm{1}_Y^\top E_k \mathbbm{1}_Y = \mathbbm{1}_Y^* E_k^* E_k \mathbbm{1}_Y = ||E_k \mathbbm{1}_Y||^2.
\end{equation}
Thus, the dual distribution is real and non-negative. Furthermore, the case $a'_k = 0$ occurs if and only if $\mathbbm{1}_Y$ is orthogonal to $V_k$.\par
As both $A_0,\ldots,A_d$ and $E_0,\ldots,E_d$ are bases of $\mathbb{A}$, there are numbers $P_i(k),Q_k(i) \in \mathbb{C}$ such that
\[
    A_i = \sum_{k=0}^d P_i(k) E_k \quad \text{ and } \quad E_k = \frac{1}{|X|}\sum_{i=0}^d Q_k(i) A_i.
\]
Since the matrices $E_k$ are pairwise orthogonal idempotents, the first equation shows that $P_i(k)$ is an eigenvalue of $A_i$ for every $i \in \{0,\ldots,d\}$. The matrix $P = (P_j(i))_{i,j=0,\ldots,d}$ is called the \emph{matrix of eigenvalues} of the association scheme. Writing $Q = (Q_j(i))_{i,j=0,\ldots,d}$, we find that $PQ = |X|I$. The matrix $Q$ is called the \emph{matrix of dual eigenvalues} of the association scheme.

Using the dual eigenvalues, we can give another formula for the dual distribution of a subset $Y$. It is the $Q$-transform of the inner distribution, that is
\begin{equation}\label{eqn:dual_distr_Q_transform}
    a'_k = \sum_{i=0}^n Q_k(i)a_i.
\end{equation}
Thus, knowledge of the dual eigenvalues is crucial for computing the dual distribution.\par
The main reason why the theory of association schemes is such a powerful tool in combinatorics, is the observation that interesting combinatorial structures can often be characterised using the inner or dual distribution. Delsarte \cite{Del1973} calls these objects \emph{cliques} and \emph{designs}, respectively.

\begin{Definition}
Let $Y$ be a non-empty subset of $X$ with inner distribution $(a_0,\ldots,a_d)$ and dual distribution $(a'_0,\ldots,a'_d)$.
\begin{enumerate}
\item[(a)]{
Let $D \subseteq \{1,\ldots,d\}$. We call $Y$ a \emph{$D$-clique} or a \emph{$D$-code} if $a_d = 0$ for every $d \not\in D$.
}
\item[(b)]{
Let $T \subseteq \{1,\ldots,d\}$. We call $Y$ a \emph{$T$-design} if $a'_t = 0$ for every $t \in T$.
}
\end{enumerate}
\end{Definition}

To match the definition of $T$-design with the expression we used to describe our main theorem, the reader will need the customary notion of the \emph{dual degree set} of a set. The dual degree
set of $Y$ is the set of indices $j\in\{1,\ldots,d\}$ such that $E_j \mathbbm{1}_Y\ne 0$.
Therefore, $Y$ is a $T$-design if and only if its dual degree set is disjoint from $T$.

\subsection{Group algebras}

The theory of complex group algebras will prove useful to derive results about association schemes. Denote by $\mathbb{C}[G]$ the complex group algebra of $G$. It is the vector space $\mathbb{C}^{G}$ of complex functions on $G$ together with convolution as multiplication. The group algebra $\mathbb{C}[G]$ is generated by the basis $(e_\sigma)_{\sigma \in G}$. We can think of a function $f:G \to \mathbb{C}$ as the formal sum
\[
    f = \sum_{\sigma \in G} f(\sigma) e_\sigma.
\]
Going the other way, we can think of an element $f \in \mathbb{C}[G]$ as a function from $G$ to $\mathbb{C}$ where $f(\sigma)$ is the coefficient of $e_\sigma$ when $f$ is written in terms of the basis $(e_\sigma)_{\sigma \in G}$. Since every character $\chi$ of $G$ is a function from $G$ to $\mathbb{C}$, every character is an element of the group algebra $\mathbb{C}[G]$. The multiplication of two elements $f,g \in \mathbb{C}[G]$ is defined as
\[
    (f \ast g)(\sigma) = \sum_{\tau \in G} f(\sigma\tau^{-1})g(\tau).
\]
Let $H$ be a subgroup of $G$. Denote by $e_{H}$ the element
\[
    e_{H} = \frac{1}{|H|}\sum_{h \in H} e_{h}.
\]
Then $e_H\ast e_H = e_H$ and $e_H$ is an idempotent of the group algebra. As a function from $G$ to $\mathbb{C}$, we have $e_H(\sigma) = |H|^{-1}$ if $\sigma \in H$ and $e(\sigma) = 0$ otherwise. Inside $\mathbb{C}[G]$, we can investigate the subalgebra $e_H \ast \mathbb{C}[G] \ast e_H$. As convolution of a function $f: G \to \mathbb{C}$ with $e_H$ roughly corresponds to averaging $f$ over $H$ (from the left or from the right), we find that $e_H \ast \mathbb{C}[G] \ast e_H$ is the algebra of functions $f: G \to \mathbb{C}$ such that
\[
    f(h\sigma h') = f(\sigma) \quad \text{ for all } h,h' \in H,
\]
that is, the functions that are both left- and right-invariant under multiplication by elements of $H$. Thus, $e_H \ast \mathbb{C}[G] \ast e_H$ is the algebra of functions from $G$ to $\mathbb{C}$ that are constant on the double cosets $H \sigma H$ of $G$. This algebra will be denoted by $\mathbb{C}[H \backslash G /H]$. It inherits the inner product $\langle \cdot,\cdot \rangle$ from $\mathbb{C}[G]$.

\subsection{Gelfand Pairs}

We refer the reader to \cite[Ch. VII.1]{Mac1995} for an introduction to Gelfand pairs. A (finite) \emph{Gelfand pair} is a pair of a group $G$ and a subgroup $H \le G$ such that the group algebra $\mathbb{C}[H \backslash G / H]$ is commutative.  Equivalently, $(G,H)$ is a Gelfand pair if and only if the induced character $1\uparrow_H^G$ of $H$ is multiplicity-free (by Schur's Lemma). We will only work with finite Gelfand pairs, hence the following theory assumes that $G$ is finite.

Every Gelfand pair $(G,H)$ gives rise to an association scheme by considering the transitive action of $G$ on the (left) cosets of $H$, and taking the orbits of $G$ on pairs for the relations. Letting $G = S_{2n}$ and $H = S_2 \wr S_n$, we have that
\[
    1\uparrow_{S_2 \wr S_n}^{S_{2n}} = \sum_{\lambda \vdash n} \chi^{2\lambda},
\]
hence $(S_{2n},S_2 \wr S_n)$ is a Gelfand pair. We refer the reader to \cite[Ch. VII.2]{Mac1995} for an excellent treatment of this Gelfand pair. In spirit with the theory of Coxeter groups, we denote the subgroup $S_2 \wr S_n$ by $B_n$. We have that $|B_n| = 2^n n!$. The subgroup $B_n$ is not a normal subgroup of $S_{2n}$ so $S_{2n}/B_n$ is not a group and hence we cannot work directly with characters and permutation representations. However, every Gelfand pair has so called \emph{zonal spherical functions} that behave very similarly to irreducible characters, and many calculations with irreducible characters in a group can be mimicked with zonal spherical functions in a Gelfand pair.

The zonal spherical functions of the Gelfand pair $(S_{2n},B_n)$ are functions from $S_{2n}$ to the complex numbers. 
They are indexed by the partitions of $n$. 
In the group algebra $\mathbb{C}[S_{2n}]$, the zonal spherical functions $\omega^\lambda$ are given by
\[
    \omega^\lambda = \chi^{2\lambda}\ast e_{B_n} = e_{B_n} \ast \chi^{2\lambda}
\]
where $\chi^{2\lambda}$ is the irreducible character of $S_{2n}$ indexed by the partition $2\lambda$. From the definition, it is immediate that the zonal spherical functions are constant on the double cosets $B_n \sigma B_n$ so we find that $\omega^\lambda \in \mathbb{C}[B_n \backslash S_{2n} / B_n]$. The double cosets are indexed by partitions of $n$ and we denote the value of the zonal spherical function $\omega^\lambda$ on the double coset of type $\rho$ by $\omega^\lambda_\rho$. Explicitly, the value of $\omega^\lambda$ on an element $\sigma \in S_{2n}$ is given by
\[
    \omega^\lambda (\sigma) = \frac{1}{|B_n|} \sum_{b \in B_n} \chi^{2\lambda}(\sigma b).
\]
Following Macdonald \cite[Ch. VII.2]{Mac1995}, we use the inner product
\[
    \langle f,g \rangle = \sum_{\sigma \in S_{2n}} f(\sigma)\overline{g(\sigma)}
\]
for $f,g \in    \mathbb{C}[B_n \backslash S_{2n} / B_n]$, omitting the scaling by $|S_{2n}|$. The zonal spherical functions have the following properties.

\begin{Theorem}[{cf. \cite[Ch. VII (1.4)]{Mac1995}}]
Let $n \in \mathbb{N}$ and $\lambda,\mu \vdash n$. Then,
\begin{enumerate}[(i)]
    \item {
        $\omega^{(n)} = 1$
    };
    \item {
        $\omega^\lambda(\sigma^{-1}) = \omega^\lambda(\sigma)$ for all $\sigma\in S_{2n}$
    };
    \item {
        $\langle \omega^\lambda,\omega^\mu \rangle = \delta_{\lambda\mu} (2n)!\chi^{2\lambda}(1)^{-1}$
    };
    \item {
        $\omega^\lambda \ast \omega^\mu = \delta_{\lambda\mu} c_\lambda \omega^\lambda$ for some $c_\lambda \neq 0$
    };
    \item{
        The zonal spherical functions $(\omega^\lambda)_{\lambda \vdash n}$ form an orthogonal basis of $\mathbb{C}[B_n \backslash S_{2n} / B_n]$.
    }
\end{enumerate}    
\end{Theorem}

So the zonal spherical functions can be thought of as analogues of the irreducible characters of a group.

\subsection{The perfect matching association scheme}

We now describe the association scheme obtained from the Gelfand pair $(S_{2n},B_n)$ in more detail.
(Some of what we cover below can be found in \cite[\S4]{Lindzey}.)
The base set of the association scheme is the coset space $S_{2n} /B_n$. As $B_n$ is the stabiliser of a perfect matching of $2n$ points, we find that the cosets of $B_n$ are in bijection with the perfect matchings of the complete graph on the vertices $1,\ldots,2n$. We can thus think of cosets of $B_n$ as perfect matchings and vice versa.

Denote by $m^*$ the perfect matching of the vertices $2i-1$ and $2i$ for $i = 1, \ldots,n$. This can be thought of as a \emph{base matching} or \emph{identity matching}. It corresponds to the coset of the identity, so $m^\ast = B_n$. The relations of the scheme are indexed by partitions of $n$. Two perfect matchings $m,\widetilde{m}$ are in relation $R_\lambda$ (denoted by $d(m,\widetilde{m}) = \lambda)$ if the (disjoint) union of $m$ and $\widetilde{m}$ consists of even cycles of lengths $2\lambda_1,2\lambda_2,\dots$. Figure \ref{ex:perfect matchings} gives an example.

The set
\[
    \Omega_\lambda = \{ \sigma B_n \; | \; d(m^*,\sigma B_n) = \lambda \}
\]
is called the \emph{$\lambda$-sphere}. We have that
\[
    d(\sigma B_n,\tau B_n) = \lambda \,\Longleftrightarrow\, d(B_n,\sigma^{-1}\tau B_n) = \lambda \,\Longleftrightarrow\, \sigma^{-1}\tau B_n \in \Omega_\lambda.
\]
If $\sigma \in S_{2n}$ such that $\sigma B_n \in \Omega_\lambda$, we say that $\sigma$ has \emph{coset type} $\lambda$. This is justified by the fact that $\sigma$ and $\tau$ have the same coset type if and only if $\tau \in B_n \sigma B_n$ (see \cite[Ch. VII (2.1)]{Mac1995}). The double cosets are indexed by partitions of $n$. The reader should note that the zonal spherical functions are constant on the $\lambda$-spheres just like the irreducible characters of a group are constant on the conjugacy classes.
\begin{center}
\begin{figure}[!ht]
\begin{tikzpicture}
\coordinate (m1) at (-4,0) {};
\coordinate (union) at (0,0) {};
\coordinate (m2) at (4,0) {};

\foreach \x in {1,...,8}{
\node[draw,circle,minimum size=1.3em,inner sep=0pt] (a\x) at ($(m1)+(157.5-45*\x:1.3)$) {\x};
\node[draw,circle,minimum size=1.3em,inner sep=0pt] (b\x) at ($(m2)+(157.5-45*\x:1.3)$) {\x};
};

\draw[color=blue,ultra thick] (a1) -- (a2);
\draw[color=blue,ultra thick] (a3) -- (a6);
\draw[color=blue,ultra thick] (a4) -- (a8);
\draw[color=blue,ultra thick] (a5) -- (a7);

\draw[color=red,ultra thick] (b1) -- (b2);
\draw[color=red,ultra thick] (b3) -- (b8);
\draw[color=red,ultra thick] (b4) -- (b5);
\draw[color=red,ultra thick] (b6) -- (b7);

\foreach \x in {1,...,8}{
\node[draw,circle,minimum size=1.3em,inner sep=0pt] (u\x) at ($(union)+(157.5-45*\x:1.3)$) {\x};
};

\draw[color=blue,ultra thick] (u1) to [bend left] (u2);
\draw[color=blue,ultra thick] (u3) -- (u6);
\draw[color=blue,ultra thick] (u4) -- (u8);
\draw[color=blue,ultra thick] (u5) -- (u7);

\draw[color=red,ultra thick] (u1) to [bend right] (u2);
\draw[color=red,ultra thick] (u3) -- (u8);
\draw[color=red,ultra thick] (u4) -- (u5);
\draw[color=red,ultra thick] (u6) -- (u7);

\node at (0,-2) {$C_6 \cup C_2 \; \longrightarrow \;$ type $(31)$};
\end{tikzpicture}
\caption{Two perfect matchings and their union.}\label{ex:perfect matchings}
\end{figure}
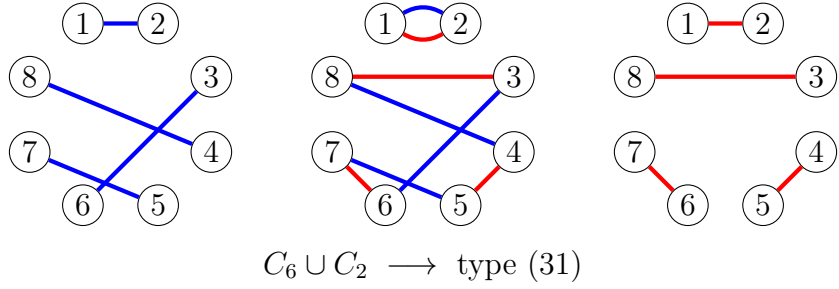
\end{center}

\subsection{Eigenvalues of the perfect matching association scheme}

Knowledge of the (dual) eigenvalues of an association scheme is of crucial importance when investigating the dual distribution and designs. Godsil and Meagher \cite{GodMea2016} have computed the eigenvalues of an association scheme obtained from a finite Gelfand pair. We repeat the result here.

\begin{Theorem}[{\cite[Cor. 13.8.2]{GodMea2016}}]\samepage
    Let $(G,H)$ be a finite Gelfand pair and $\mathbb{A}$ be the association scheme obtained from the action of $G$ on the left cosets of $H$. The matrix idempotents of $\mathbb{A}$ are $E_\phi$ where the entries are given by
    \[
        E_\phi(x_iH,x_jH) = \frac{\phi(1)}{|G|}\sum_{\sigma \in H} \phi(x_i \sigma x_j ^{-1})
    \]
    and $\phi$ is an irreducible constituent of $1 \uparrow_{H}^{G}$.
\end{Theorem}

\begin{Corollary}
    The idempotents of the perfect matching association scheme are given by
    \[
        E_\mu(x_iB_n,x_jB_n) = \frac{\chi^{2\mu}(1)}{(2n)!}\sum_{\sigma \in B_n} \chi^{2\mu}(x_i\sigma x_j^{-1}).
    \]
    The $(x_iB_n,x_jB_n)$-entry of $E_\mu$ is equal to $Q_\mu(\rho)/(2n-1)!!$ where $d(x_iB_n,x_jB_n) = \rho$, equivalently $x_i^{-1}x_j$ has coset type $\rho$. We obtain
    \[
        Q_\mu(\rho) = \chi^{2\mu}(1) \omega^\mu_\rho.
    \]
\end{Corollary}

\begin{proof}
    Notice that
    \begin{align*}
     \frac{\chi^{2\mu}(1)}{(2n)!}\sum_{\sigma \in B_n} \chi^{2\mu}(x_i\sigma x_j^{-1}) &= \frac{\chi^{2\mu}(1)}{(2n)!}\sum_{\sigma \in B_n} \chi^{2\mu}(x_j^{-1}x_i\sigma ) \\
     &=\frac{\chi^{2\mu}(1) |B_n|}{(2n)!}\omega^\mu(x_j^{-1}x_i) \\
     &= \frac{\chi^{2\mu}(1)}{(2n-1)!!}\omega^\mu(x_j^{-1}x_i).
    \end{align*}
    The value $\omega^\mu$ only depends on the coset type of $x_j^{-1}x_i$ which has the same coset type as $(x_j^{-1}x_i)^{-1}=x_i^{-1}x_j$.
\end{proof}

Using this corollary and Equation \eqref{eqn:dual_distr_Q_transform}, we obtain a formula for the dual distribution. The dual distribution of a non-empty set $Z \subseteq S_{2n}/B_n$ is given by
\[
    a'_\mu (Z) = \sum_{\rho \,\vdash\, n} Q_\mu(\rho) a_\rho(Z) = \chi^{2\mu}(1)  \sum_{\rho\,\vdash\,n} \omega^\mu_\rho a_\rho(Z).
\]
Hence,
\[
    a'_\mu(Z) = 0 \;\Longleftrightarrow\; \sum_{\rho\,\vdash\,n} \omega^\mu_\rho a_\rho(Z) = 0.
\]

\section{Designs}\label{section:designs}

In this section, we introduce a generalisation of hyperfactorisations of $K_{2n}$.

\begin{Definition}
Let $\lambda \,\vdash\, n$ and consider all set partitions of shape $2\lambda$. A set $D\neq \varnothing$ of perfect matchings is called a \emph{$\lambda$-factorisation} if there exists a constant $c>0$ such that for every set partition $P$ of shape $2\lambda$ there are exactly $c$ matchings $m \in D$ such that $m$ refines $P$. The constant $c$ is called the \emph{index} of $D$.
\end{Definition}

\begin{example}\label{example:lambda-factorisations}
    For every $n$, an $(n-1,1)$-factorisation with index 1 is just a 1-factorisation,
    and an $(n-2,1,1)$-factorisation is a hyperfactorisation. 
    Cameron's examples \cite{Cam1976} arising from hyperovals of projective planes are $(n-2,1,1)$-factorisations of index 1 where
    $2n=2^a+2$ for some integer $a\ge 3$.
    Indeed, Cameron's $s$-$(2,2n)$ partition systems are precisely
    the $(n-s,1,1,\ldots,1)$-factorisations with index 1.\\
    For every $n$ and $t \le n/2$, an $(n-t,t)$-factorisation is a set $D \neq \varnothing$ of perfect matchings such that for every subset $S \subseteq \{1,\ldots,2n\}$ of size $2t$, there are a constant number of perfect matchings in $D$ that have $t$ edges that all lie inside of $S$. Hence, every subset of size $2t$ is covered constantly often by the perfect matchings of $D$.
\end{example}

A hyperoval $\mathcal{O}$ of a projective plane $\Pi$ of order $q$ is a set
of $q+2$ points such that every line of $\Pi$ intersects $\mathcal{O}$ in 0 or 2 points.
Each point $P$ not in $\mathcal{O}$ defines a fixed-point-free permutation of $\mathcal{O}$,
taking an element $X$ to the unique element $X^P$ of $\mathcal{O}\backslash\{X\}$ lying on the line
spanned by $P$ and $X$. The set of $q^2-1$ permutations of $\mathcal{O}$ arising this
way gives rise to an \emph{abstract hyperoval}. Indeed, each permutation
is a complete product of 2-cycles, and so the set of perfect matchings that arise is identical
to Cameron's construction. However, not all abstract hyperovals arise from a hyperoval.

\begin{example}\label{example:abstracthyperovals}
Abstract hyperovals were introduced by J. G. Thompson in \cite{Thompson}. An
\emph{abstract hyperoval} $A(X)$ on a finite set $X$ is a set of fixed-point-free involutions on $X$ with the following property: for any four elements $a,b,c,d\in X$, there is a unique $u\in A(X)$ such that $a^u=b$ and $c^u=d$. If $|X|=n+2$, then we say that $A(X)$ has \emph{order $n$}, and we see from the definition that $n$ is even and $|A(X)|=n^2-1$. De Clerck \cite{DeClerck1979} showed that an abstract hyperoval of order $2s$ is equivalent to a partial geometry with parameters $\mathrm{pg}(s, 2s - 2, s - 1)$. We refer to \cite{Cooper} and \cite{Prince} for more on this connection.

Now a perfect matching defines a fixed-point-free involution, and vice versa, because a fixed-point-free involution is a complete product of disjoint 2-cycles. So if $n$ is even,
a $2$-$(2,n+2)$ partition system is a set of $n^2-1$ perfect matchings of $K_{n+2}$ such that every
pair of disjoint edges $\{a,b\}$ and $\{c,d\}$ lies in exactly one perfect matching. So we
see readily that abstract hyperovals of order $n$ and $2$-$(2,n+2)$ partition systems are equivalent 
objects. This observation generalises \cite[Theorem 7.3]{Cam1976}, and can also
be deduced from \cite[Chapter 7]{Cooper} and \cite[Lemma 1]{Faina1984}. 

For $n=2,4$ the only abstract hyperovals are obtained
by taking the set of all fixed-point-free involutions. There is no example for
$n=6$ (see \cite{Mathon}) and there are precisely two examples for $n=8$ (see also \cite{Mathon}).
So there are only two $2$-$(2,10)$ partition systems up to equivalence, and they have
stabilisers $\mathrm{P\Gamma L}(2,8)$ and $M_9:C_3$ respectively.
In our language, there are two $(3,1,1)$-factorisations of index 1. We will look at related
examples in Section \ref{section:smalln}. We remark that Thompson \cite{Thompson} explored the
use of the representation theory of $S_{2n}$ to derive restrictions on the
existence of abstract hyperovals, and we take this viewpoint further
and apply this theory to $\lambda$-factorisations. 
\end{example}

Our goal is to show the following theorem that characterises $\lambda$-factorisations as $T$-designs in an association scheme.

\begin{Theorem}\label{theorem:main1}
    Let $D \subseteq S_{2n} / B_n$ be a non-empty set of perfect matchings and $(a'_{\mu})_{\mu \,\vdash\,n}$ be its dual distribution. Then
    \[
        D \text{ is a } \lambda\text{-factorisation } \;\Longleftrightarrow \; a'_\mu = 0 \text{ for all } \mu \vdash n \text{ with } \lambda \ltriangle \mu \neq (n).
    \]
\end{Theorem}

\subsection{Antidesigns}

We will prove our main theorem (Theorem \ref{theorem:main1}) using the notion of \emph{antidesigns}.

\begin{Definition}
    Let $(X,R)$ be an association scheme and $A,D\subseteq X$ be non-empty subsets of $X$ with dual distributions $(a^{'A}_k)_{k=0,\ldots,d}$ and $(a^{'D}_k)_{k=0,\ldots,d}$, respectively.
    \begin{enumerate}
        \item{
            The pair $(A,D)$ is called \emph{design-orthogonal}, if $a^{'A}_k \cdot a^{'D}_k = 0$ for every $k>0$.
        }
        \item{
            Let $T = \{k \in \{1,\ldots,d\} \mid a^{'A}_k \neq 0\}$. Then we call $A$ a \emph{$T$-antidesign.}
        }
    \end{enumerate}
\end{Definition}

It is clear from the definition that if the pair $(A,D)$ is design-orthogonal and $A$ is a $T$-antidesign, then $D$ must be a $T$-design. Hence, we can characterise $T$-designs using $T$-antidesigns.
It turns out (see \cite{Delsarte77}) that if an association scheme admits an automorphism group that acts transitively on each relation $R_i$, then a pair $(A,D)$ is design-orthogonal if and only if the intersection $|D \cap \varphi(A)|$ is constant when $\varphi$ varies over all automorphisms of $(X,R)$.

Applied to the association scheme of perfect matchings, we find that $S_{2n}$ acts transitively on each of the relations $R_\mu$. Let $\lambda \vdash n$ and let $P_\lambda$ be a set partition of shape $2\lambda$. Consider the set $A$ of all perfect matchings that refine $P_\lambda$. A $\lambda$-factorisation $D$ of index $c$ intersects $A$ in precisely $c$ elements. Note that $S_{2n}$ acts transitively on the set partitions of shape $2\lambda$ and that, from the definition of a $\lambda$-factorisation, $|D \cap \varphi(A)|$ is constant for all $\varphi \in S_{2n}$ . Thus, the pair $(A,D)$ is design-orthogonal. If $A$ is a $T$-antidesign, then $D$ is a $T$-design. Hence, it suffices to compute the set $T$ such that $A$ is a $T$-antidesign.

We find that $A$ is the set of left cosets of $B_n$ with representatives in $S_{2\lambda}B_n$, where $S_{2\lambda}$ is a Young subgroup. The size of $A$ in the association scheme $S_{2n}/B_n$ is
\[
    |A| = \frac{|S_{2\lambda} B_n|}{|B_n|} = \frac{|S_{2\lambda}|}{|S_{2\lambda} \cap B_n|} =  \frac{|S_{2\lambda}|}{|B_\lambda|} = \frac{(2\lambda)!}{2^n \lambda!}
\]
where $B_\lambda = B_{\lambda_1} \times B_{\lambda_2} \times\dots\times B_{\lambda_{l(\lambda)}}$. We have that $|B_\lambda| = 2^n \lambda!$.

First, we compute the inner distribution of $A$.
\begin{Lemma}\label{lem:inner}
    The inner distribution $a_\rho(A)$ is given by
    \[
        a_\rho(A) = \frac{1}{2^n \lambda!} |S_{2\lambda} \cap B_n x_\rho B_n|
    \]
    where $x_\rho \in S_{2n}$ is any element of coset type $\rho$.
\end{Lemma}

\begin{proof}
    First, recall that $d(\sigma B_n,\tau B_n) = \rho$ if and only if the coset type of $\sigma^{-1} \tau$ is $\rho$. Since every matching in $A$ corresponds to a left coset $\sigma B_n$ and all elements of $\sigma B_n$ have the same coset type, we get the number of pairs $(\sigma B_n,\tau B_n)$ with $d(\sigma B_n,\tau B_n) = \rho$ by counting the number of pairs $(\sigma,\tau)\in (S_{2\lambda}B_n)^2$ such that $\sigma^{-1}\tau$ is of coset type $\rho$ and dividing the result by $|B_n|^2$. Write
    \[
        S_{2\lambda}B_n = \bigcup_{r \in R} S_{2\lambda}r
    \]
    where $R \subseteq B_n$ is a complete set of representatives of the decomposition of $S_{2\lambda}B_n$ into right cosets of $S_{2\lambda}$.  Then every element of $z \in A$ can be uniquely written as $z=\sigma r$ with $\sigma \in S_{2\lambda}$ and $r \in R$.
    Note that $|R|=|B_n:S_{2\lambda}\cap B_n|=
    |B_n:B_\lambda|=|A||B_n|/|S_{2\lambda}|$.
    Counting the pairs in the symmetric group gives
    \begin{align*}
       &\;|\{(\sigma,\tau) \in (S_{2\lambda}B_n)^2 :\text{coset type}(\sigma^{-1}\tau) = \rho \}|\\
       =&\; |\{(\sigma,r,\tau,s) \in (S_{2\lambda}\times R)^2 : (\sigma r)^{-1} \tau s \in B_n x_\rho B_n \}|\\
       =&\; |\{(\sigma,r,\tau,s) \in (S_{2\lambda}\times R)^2 : r^{-1}\sigma^{-1} \tau s \in B_n x_\rho B_n \}|\\      
       =&\; |R|^2|\,\{(\sigma,\tau) \in S_{2\lambda}^2 : \sigma^{-1} \tau \in B_n x_\rho B_n \}|\\
       =&\; |S_{2\lambda}|\,|R|^2|\,\{\sigma \in S_{2\lambda} : \sigma \in B_n x_\rho B_n \}|\\
       =&\;|S_{2\lambda}||R|^2 |S_{2\lambda} \cap B_n x_\rho B_n|\\
       =&\;\frac{|A|^2|B_n|^2}{|S_{2\lambda}|} |S_{2\lambda} \cap B_n x_\rho B_n|\\
       =&\;\frac{|A||B_n|^2}{|B_\lambda|} |S_{2\lambda} \cap B_n x_\rho B_n|.
    \end{align*}
    Dividing by $|B_n|^2$ gives the number of pairs in the scheme $S_{2n}/B_n$ and dividing by $|A|$ gives the inner distribution. 
\end{proof}

Now let $\theta^\lambda = e_{B_n} \ast {1}_{2\lambda} \ast e_{B_n} \in \mathbb{C}[B_n \backslash S_{2n} / B_n]$ where $1_{2\lambda}$ is the trivial character of $S_{2\lambda}$. In $\mathbb{C}[S_{2n}]$, $1_{2\lambda}$ is the characteristic vector of the subgroup $S_{2\lambda}$. Explicitly, we have
\[
    \theta^\lambda(\sigma) = (e_{B_n} \ast {1}_{2\lambda} \ast e_{B_n})(\sigma) = \frac{1}{|B_n|^2}\sum_{a,b \in B_n} 1_{2\lambda}(a\sigma b).
\]
 We can decompose $\theta^\lambda$ in terms of the orthogonal basis of $\mathbb{C}[B_n \backslash S_{2n} / B_n]$ given by the zonal spherical functions $\omega^\mu$. Let $x_\rho \in S_{2n}$ be any element of coset type $\rho$, recall that $\omega^\mu$ and $\theta^\lambda$ are constant on the double cosets and calculate
\begin{align*}
    \langle \omega^\mu,\theta^\lambda \rangle &= \sum_{\sigma \in S_{2n}} \omega^\mu(\sigma)\overline{\theta^\lambda(\sigma)} = \sum_{\rho\,\vdash\,n} |B_n x_\rho B_n| \omega^{\mu}_\rho \theta^\lambda(x_\rho)\\
    &= \frac{1}{|B_n|^2} \sum_{\rho \,\vdash\,n} |B_n x_\rho B_n|\,\omega^\mu_\rho \sum_{a,b \in B_n} 1_{2\lambda}(a x_\rho b).
\end{align*}
So it remains to calculate the inner sum. We find
\[
    \sum_{a,b \in B_n} 1_{2\lambda}(a x_\rho b)  = \sum_{a \in B_n} \sum_{b \in B_n} 1_{2\lambda}(ax_\rho b) = \sum_{a \in B_n} |S_{2\lambda} \cap ax_\rho B_n|.
\]
We have $ax_\rho B_n = bx_\rho B_n \,\Longleftrightarrow\, a^{-1}b \in x_\rho B_n x_\rho^{-1}$. Since $a,b \in B_n$, we find that $ax_\rho$ and $bx_\rho$ are representatives of the same coset if and only if $a^{-1} b \in B_n \cap x_\rho B_n x_\rho^{-1}$, equivalently $ a(B_n \cap x_\rho B_n x_\rho^{-1})= b (B_n \cap x_\rho B_n x_\rho^{-1})$. Let $R \subseteq B_n$ be a system of representatives of the left cosets of $B_n \,/\,(B_n \cap x_\rho B_n x_\rho^{-1})$. Then we find
\begin{align*}
    \sum_{a \in B_n} |S_{2\lambda} \cap ax_\rho B_n| &= \sum_{r \in R} \sum_{\sigma \in B_n \cap x_\rho B_n x_\rho^{-1}} |S_{2\lambda} \cap r\sigma x_\rho B_n| \\
    &= \sum_{r \in R} \sum_{\sigma \in B_n \cap x_\rho B_n x_\rho^{-1}} |S_{2\lambda} \cap rx_\rho B_n|\\
    &= \sum_{r \in R} |B_n \cap x_\rho B_n x_\rho^{-1}|\,|S_{2\lambda} \cap rx_\rho B_n|\\
    &= |B_n \cap x_\rho B_n x_\rho^{-1}|\sum_{r \in R} |S_{2\lambda} \cap rx_\rho B_n|\\
    &= |B_n \cap x_\rho B_n x_\rho^{-1}|\,|S_{2\lambda} \cap B_n x_\rho B_n|.
\end{align*}
We have that (cf. \cite[Ch. VII (2.3)]{Mac1995})
\begin{align*}
    |B_n x_\rho B_n| = \frac{|B_n|^2}{|B_n \cap x_\rho B_n x_\rho^{-1}|}.
\end{align*}
Putting it all together, we obtain
\begin{align*}
    \langle \omega^\mu,\theta^\lambda \rangle &= \frac{1}{|B_n|^2} \sum_{\rho \,\vdash\,n} |B_n x_\rho B_n|\,\omega^\mu_\rho \sum_{a \in B_n} |S_{2\lambda} \cap ax_\rho B_n|\\
    &= \frac{1}{|B_n|^2} \sum_{\rho \,\vdash\,n} |B_n x_\rho B_n|\,\omega^\mu_\rho \, |B_n \cap x_\rho B_n x_\rho^{-1}|\,|S_{2\lambda} \cap B_n x_\rho B_n|\\
    &= \sum_{\rho \,\vdash\,n} \omega^\mu_\rho \,|S_{2\lambda}\cap B_n x_\rho B_n| = 2^n \lambda! \sum_{\rho \,\vdash\,n} \omega^\mu_\rho a_{\rho}(A) \quad (\text{Lemma \ref{lem:inner}}) \\
    & = \frac{2^n \lambda!}{\chi^{2\mu}(1)}\,\chi^{2\mu}(1) \sum_{\rho \,\vdash\,n} \omega^\mu_\rho a_{\rho}(A) = \frac{2^n \lambda!}{\chi^{2\mu}(1)} \, a'_\mu(A).
\end{align*}
Hence, $\langle \omega^\mu,\theta^\lambda \rangle$ gives the dual distribution of $A$ up to a non-zero constant.

Indeed, $\langle \omega^\mu,\theta^\lambda \rangle$ is a positive multiple of the coefficient $u_{\mu,\lambda}$ of the monomial symmetric function $m_\lambda$ in the expansion of the zonal polynomial $Z_\mu$ (see the proof of \cite[Ch. VII (2.22)]{Mac1995}). This coefficient is 0 unless $\mu \rtriangle \lambda$. Now by \cite[Ch. VII (2.23)]{Mac1995}, $Z_\mu$ is the \emph{Jack symmetric function} with parameter $\alpha = 2$. By \cite[Ch. VI (10.13),(10.15)]{Mac1995}, $u_{\mu,\lambda}>0$ whenever $\mu \rtriangle \lambda$. It follows that
\[
    \langle \omega^\mu,\theta^\lambda \rangle \neq 0 \;\Longleftrightarrow \; \mu \rtriangle \lambda.
\]

We can now prove Theorem \ref{theorem:main1}.

\begin{proof}[Proof of Theorem \ref{theorem:main1}]
    A $\lambda$-factorisation $D$ is a $T$-design where $T$ is the set of partitions such that the set $A$ of left cosets of $B_n$ with representatives in $S_{2\lambda}B_n$ is a $T$-antidesign. By the calculation above, the dual distribution of $A$ is
    \(
        a'_\mu (A) = C\langle \omega^\mu,\theta^\lambda \rangle
    \)
    for some non-zero constant~$C$. Since $\langle \omega^\mu,\theta^\lambda \rangle \neq 0$ if and only if $\mu \rtriangle \lambda$, the result follows.
\end{proof}

By characterising $\lambda$-factorisations in terms of the dual distribution, we can compare $\lambda$-factorisations of different types.

\begin{Theorem}\label{theorem:comparison of strengths}
    Let $\lambda,\mu \,\vdash n$ and $\varnothing \neq D \subseteq S_{2n}/B_n$ be a $\lambda$-factorisation. If $\mu \rtriangle \lambda$, then $D$ is also a $\mu$-factorisation.
\end{Theorem}

\begin{proof}
    Follows immediately from Theorem \ref{theorem:main1} and the fact that the dominance order $\ltriangle$ is transitive.
\end{proof}

\begin{Remark}
    Notice that the converse of Theorem \ref{theorem:comparison of strengths} does not necessarily
    hold if $\lambda$ is not `small' enough. Consider $n=6$ and suppose $D$ is a $(3,2,1)$-factorisation. Then $D$ is also a $\mu$-factorisation for $\mu\in\{(3,3),(4,1,1),(4,2),(5,1)\}$. Now $D$ is a $\lambda$-factorisation where $\lambda=(4,1,1)$, but we also have that $D$ is a $\mu$-factorisation where $\mu=(3,3)$; but $\lambda$ and $\mu$
    are incomparable in the dominance order.
\end{Remark}

Theorem \ref{theorem:comparison of strengths}
gives divisibility constraints on the indices reminiscent of the arithmetic conditions on the parameters of block designs.

\begin{Theorem}\label{thm:comparison_general}
    Let $\lambda,\mu \,\vdash n$ with $\mu \rtriangle \lambda$ and $\varnothing \neq D \subseteq S_{2n}/B_n$ be a $\lambda$-factorisation of index $c_\lambda$. Then $D$ is a $\mu$-factorisation of index $c_\mu$ where
    \[
        c_\mu = c_\lambda \,\frac{(2\mu_1-1)!!(2\mu_2 -1)!!\dots}{(2\lambda_1-1)!!(2\lambda_2 -1)!!\dots}.
    \]
    Hence $(2\lambda_1 -1)!!(2\lambda_2-1)!!\dots$ divides $c_\lambda (2\mu_1-1)!!(2\mu_2 -1)!!\dots$ for all $\mu \rtriangle \lambda$.
\end{Theorem}

\begin{proof}
    We count the pairs $(m,P)$ of matchings $m \in D$ that refine a set partition $P$ of shape $2\lambda$ in two ways. First, we count the number of set partitions of $\{1,\ldots,2n\}$ of shape $2\lambda$, and then we count the number of set partitions that one matching $m$ refines.\par
    Denote by $m_i$ the number of parts of $\lambda$ of size $i$. Then the number of set partitions of shape $2\lambda$ is
    \[
        \frac{(2n)!}{\prod\limits_{i \ge 1,m_i>0} (2i)!^{m_i} m_i!}
    \]
    because the parts of a set partition of shape $2\lambda$ have even size and the parts of size $2i$ are stabilised by the wreath product $S_{2i} \wr S_{m_i}$.
    
    Now consider any perfect matching $m$. Counting the number of set partitions that $m$ refines comes down to counting in how many ways the $n$ edges of $m$ can be grouped into parts according to $\lambda$. We find that there are
    \[
        \frac{n!}{\prod\limits_{i \ge 1,m_i>0} i!^{m_i} m_i!}
    \]
    set partitions of shape $2\lambda$ that $m$ refines. Hence, double counting the pairs $(m,P)$ of matchings $m \in D$ that refine a set partition $P$ of shape $2\lambda$, we obtain
    \[
        c_\lambda \,\frac{(2n)!}{\prod\limits_{i \ge 1,m_i>0} (2i)!^{m_i} m_i!} = |D| \,\frac{n!}{\prod\limits_{i \ge 1,m_i>0} i!^{m_i} m_i!} \; \Longleftrightarrow \; |D| = c_\lambda \frac{(2n)!}{n!} \prod\limits_{i \ge 1,m_i>0} \left(\frac{i!}{(2i)!}\right)^{m_i}.
    \]
    Now $(2n)!/n! = 2^n(2n-1)!!$, the powers of $2$ cancel and we can rewrite the product as
    \[
        |D| = c_\lambda (2n-1)!! \prod\limits_{i = 1}^{l(\lambda)} \frac{1}{(2\lambda_i-1)!!}.
    \]
    By Theorem \ref{theorem:comparison of strengths}, we find that $D$ is also a $\mu$-factorisation, so we can also compute the size of $D$ in terms of $\mu$. Equating both counts gives
    \[
        c_\lambda (2n-1)!! \prod\limits_{i = 1}^{l(\lambda)} \frac{1}{(2\lambda_i-1)!!} = c_\mu (2n-1)!! \prod\limits_{i = 1}^{l(\mu)} \frac{1}{(2\mu_i-1)!!} \; \Longleftrightarrow \; c_\mu = c_\lambda \frac{(2 \mu_1 -1)!!(2\mu_2 -1)!!\ldots}{(2 \lambda_1 -1)!!(2\lambda_2 -1)!!\ldots}.
    \]
    Since $c_\mu$ is an integer, the result follows.
\end{proof}

The formula in Theorem \ref{thm:comparison_general} also reveals a simple way to obtain the size of
a $\lambda$-factorisation of index $c_\lambda$ (i.e., use $\mu=(n)$):
\[
  c_\lambda\, \frac{(2n-1)!!}{(2\lambda_1-1)!!(2\lambda_2 -1)!!\dots}.
\]

The following is a simple yet powerful special case of Theorem \ref{thm:comparison_general}.
\begin{Corollary}\label{cor:k_and_l}
    Let $\lambda \vdash n$, $\lambda\neq (n)$, and let $D \neq \varnothing$ be a $\lambda$-factorisation of index $c_\lambda$. If $k,l$ with $k \le l$ are distinct parts of $\lambda$ (which may be the same if $\lambda$ has multiple parts of the same size), then $2k-1$ divides $(2l+1)c_\lambda$. In particular, if $D$ is of index 1, then $2k-1$ divides $2l+1$.
\end{Corollary}

\begin{proof}
   Write $\lambda = \widetilde{\lambda} \cup (l,k) $ and let $\mu = \widetilde{\lambda} \cup (l+1,k-1)$. Then we have $\lambda \ltriangle \mu$.  Now use Theorem \ref{thm:comparison_general} to find that
   \[
        c_\mu = c_\lambda \,\frac{(2(k-1)-1)!!(2(l+1)-1)!!}{(2k-1)!!(2l-1)!!} = c_\lambda \frac{2l+1}{2k-1}.\qedhere
   \]
\end{proof}

This severely restricts the possible partitions $\lambda$ such that a $\lambda$-factorisation of index 1 exists. Applying Theorem \ref{thm:comparison_general} to partitions with two parts, we get the following:

\begin{Corollary}\label{cor:comparison_two_parts}
    Let $t \le n/2$ and $\varnothing \neq D \subseteq S_{2n}/B_n$ be an $(n-t,t)$-factorisation of index $c_t$. Then $D$ is an $(n-t+1,t-1)$-factorisation of index $c_{t-1}$ where
    \[
        c_{t-1} = c_t\,\frac{2n-2t+1}{2t-1} = c_t\,\frac{2n}{2t-1} -c_t.
    \]
\end{Corollary}

\begin{proof}
    Use Theorem \ref{thm:comparison_general} with $\lambda = (n-t,t)$ and $\mu = (n-t+1,t-1)$, which satisfy $\mu \rtriangle \lambda$.
\end{proof}

    Notice that $c_{t-1}$ is an integer if and only if $2t-1 \,|\, c_t\cdot2n$. This gives the following non-existence result.

\begin{Corollary}\label{cor:comparison_t_factorisations}
    For every $n\ge 4$, an $(n-2,2)$-factorisation of index coprime to 3 (in particular index 1) can only exist if $n \equiv 0 \pmod{3}$.
\end{Corollary}

\begin{proof}
    Use Corollary \ref{cor:comparison_two_parts} for $t=2$.
\end{proof}

Note that Corollary \ref{cor:comparison_two_parts} can be applied repeatedly, i.e., an $(n-t,t)$-factorisation is an $(n-s,s)$-factorisation for every $s= t-1,t-2, \ldots,1$. For example, an $(n-3,3)$-factorisation of index 1 (where $n \ge 6$) is an $(n-2,2)$-factorisation of index $\frac{2n-5}{5}$, so $n \equiv 0 \pmod 5$. Now an $(n-2,2)$-factorisation of index $\frac{2n-5}{5}$ is an $(n-1,1)$-factorisation of index $\frac{2n-5}{5}\cdot\frac{2n-3}{3}$, so $(2n-5)(2n-3) \equiv 0 \pmod 3$, hence $n \equiv 0,1 \pmod 3$. So an $(n-3,3)$-factorisation of index 1 can only exist if $n \equiv 0,10 \pmod{15}$. In general, we get the following divisibility criteria.

\begin{Theorem}\label{theorem:twoparts}
    Let $t \le n/2$. If an $(n-t,t)$-factorisation of index 1 exists, then
    \[
        (2t-1)(2t-3)\ldots (2k+1) \text{ divides } (2n-(2t-1))(2n-(2t-3))\ldots(2n-(2k+1))
    \]
    for every $k=t-1,t-2,\ldots,1$. Equivalently,
    \[
        \frac{(2(n-k)-1)!!(2k-1)!!}{(2(n-t)-1)!!(2t-1)!!}
    \]
    is an integer for every $k=t-1,t-2,\ldots,1$.
\end{Theorem}

\begin{proof}
    Apply Theorem \ref{thm:comparison_general} or Corollary \ref{cor:comparison_two_parts} to the chain $(n-t,t) \ltriangle (n-t+1,t-1) \ltriangle \ldots \ltriangle (n-2,2) \ltriangle (n-1,1)$.
\end{proof}

Note that Theorem \ref{theorem:twoparts} also holds in the case $k=0$ but the resulting divisibility condition is a weaker version of the divisibility condition for $k=1$.

For many combinatorial objects there exists a process called \emph{derivation} that builds smaller objects from bigger ones. This is also the case for $\lambda$-factorisations.

\begin{Definition}
    Let $\lambda \vdash n$, $\lambda \neq (n)$, and $D$ be a $\lambda$-factorisation. Let $k$ be a part of $\lambda$ and $S \subseteq \{1,\ldots,2n\}$ be a set of size $2k$. For a perfect matching $m \in D$, we write $m \setminus S$ for the set of edges of $m$ that are disjoint from $S$. Then the set
    \[
        D_S = \left\lbrace m \setminus S \mid m \text{ refines } \{S,[2n] \setminus S\right\rbrace \}
    \]
    is called the \emph{derivation} of $D$ at $S$.
\end{Definition}

The following theorem is immediate.

\begin{Theorem}\label{theorem:derivation}
    Let $\lambda \vdash n$, $\lambda \neq (n)$, and $D$ be a $\lambda$-factorisation of index $c$. Let $k$ be a part of $\lambda$, $S \subseteq \{1,\ldots,2n\}$ be a set of size $2k$ and write $\lambda = \widetilde{\lambda} \cup (k)$ for a partition $\widetilde{\lambda} \vdash n-k$. Then the derivation $D_S$ is a $\widetilde{\lambda}$-factorisation of index $c$.
\end{Theorem}

\begin{proof}
    Without loss of generality, we can assume that $S = \{2(n-k)+1,\ldots,2n\}$ (otherwise, permute the elements of $\{1,\ldots,2n\}$ suitably). Consider the set $P_S$ of all set partitions of $\{1,\ldots,2n\}$ of shape $2\lambda$ that contain $S$. We find
    \[
        P_S = \{ S \cup \widetilde{P} \mid \widetilde{P} \text{ is a set partition of } \{1,\ldots,2(n-k)\} \text{ of shape } 2\widetilde{\lambda} \}.
    \]
    Now let $\widetilde{P}$ be any set partition of $\{1,\ldots,2(n-k)\}$ of shape $2\widetilde{\lambda}$. Then $S\cup \widetilde{P} \in P_S$ and $S\cup \widetilde{P}$ has shape $2\lambda$. Since $D$ is a $\lambda$-factorisation, we have that there are exactly $c$ perfect matchings $m \in D$ that refine $S \cup \widetilde{P}$. All matchings $m \in D$ that refine a set partition of $P_S$ also refine the set partition $\{S,[2n]\setminus S\}$ and thus give rise to elements $m \setminus S$ of $D_S$. Hence, there are exactly $c$ perfect matchings $\widetilde{m} \in D_S$ that refine $\widetilde{P}$ and $D_S$ is a $\widetilde{\lambda}$-factorisation of index $c$.
\end{proof}

For two partitions $\lambda,\mu$, we write $\mu \preccurlyeq \lambda$ if every part of $\mu$ is a part of $\lambda$. Then Theorem \ref{theorem:derivation} implies the following.

\begin{Corollary}\label{cor:derivation}
    Let $\lambda \vdash n$ and $D$ be a $\lambda$-factorisation of index $c$. Then there exist $\mu$-factorisations of index $c$ for every partition $\mu \preccurlyeq \lambda$.
\end{Corollary}

\begin{proof}
    Use derivation multiple times.
\end{proof}

Corollary \ref{cor:derivation} implies that if a $\lambda$-factorisation of index $c$ does not exist, then there is no $\mu$-factorisation of index $c$ where all parts of $\lambda$ turn up as parts of $\mu$. Thus, knowledge of the existence or non-existence of $\lambda$-factorisations where $\lambda \vdash n$ with small $n$ gives insight into the existence for bigger values of $n$. For example, we can generalise a result of Cameron \cite[Theorem 7.6]{Cam1976}
and Mathon \cite{Mathon} that there is no $2$-$(2,8)$ partition system.

\begin{Corollary}\label{corollary:2111}
Let $n\ge 4$. There is no $(2,1,\ldots,1)$-factorisation of index 1.
In particular, there is no $(n-2)$-$(2,2n)$ partition system.
\end{Corollary}

\begin{proof}
    By Cameron \cite[Theorem 7.6]{Cam1976} and Mathon \cite{Mathon}, there is no $(2,1,1)$-factorisation of index 1. Now the result follows from Corollary \ref{cor:derivation}.
\end{proof}

Alternatively, the theory of symmetric functions can be used for a longer proof of Corollary \ref{corollary:2111}. We present it here, because it also shows that there is a Krein parameter of the association scheme on perfect matchings that is guaranteed to be 0, and this is of independent interest.

\begin{example}\label{example:2111}
Let $S$ be a putative $(2,1,\ldots,1)$-factorisation of index 1.
By Theorem \ref{theorem:main1}, the dual distribution
of $S$ is of the form $(1,x,0,\ldots,0)$ where the positive
rational number $x$ appears in the coordinate indexed by 
the partition $\mu:=(1,1,\ldots, 1)$. 
We now calculate the Krein parameter $q_{\mu\mu}^\mu$, via
the formula given in \cite[Theorem 3.6 and 3.5(i)]{BanIto1984} that uses the eigenvalues and the valencies of the association scheme.

The valency $k_\rho$ of the perfect matching scheme is the size of the $\rho$-sphere $\Omega_\rho$. This is the size of the double coset $B_n x_\rho B_n$ of type $\rho$. By \cite[Ch. VII (2.3)]{Mac1995}, it is given by the formula $|B_n|^2z_{2\rho}^{-1}$ where $z_{2\rho}$ is a non-zero number whose exact value need not concern us. The eigenvalue $P_\mu(\rho)$ of the perfect matching scheme is given by the formula $P_\mu(\rho) = k_\rho \omega^\mu_\rho$ (cf. \cite[Lemma 13.8.3]{GodMea2016}).

We will use the relation $\equiv$ to denote `up to a nonzero scalar'. Using the formula for the value of the zonal spherical function $\omega^{(1,\ldots,1)}_\rho$ given in \cite[Ch. VII Ex. 2(b)]{Mac1995}, we find
\begin{align*}
    q_{\mu\mu}^\mu &\equiv \sum_{\rho \vdash n} \frac{1}{k_\rho^2}\left(P_\mu(\rho)\right)^3
    =  \sum_{\rho \vdash n} k_\rho\left(\omega^\mu_\rho\right)^3
    \equiv  \sum_{\rho \vdash n} z_{2\rho}^{-1} \omega^{(1,\ldots,1)}_\rho \cdot \left(\omega^{(1,\ldots,1)}_\rho\right)^2\\
    &= \sum_{\rho \vdash n} z_{2\rho}^{-1} \omega^{(1,\ldots,1)}_\rho \cdot \left(\frac{(-1)^{n-l(\rho)}}{2^{n-l(\rho)}}\right)^2
    = \sum_{\rho \vdash n} z_{2\rho}^{-1} \omega^{(1,\ldots,1)}_\rho \cdot \frac{1}{2^{2n-2l(\rho)}}\\
    &= 4^{-n} \sum_{\rho \vdash n} z_{2\rho}^{-1} \omega^{(1,\ldots,1)}_\rho \cdot 4^{l(\rho)}\\
    &\equiv c_{(1,\ldots,1)}(4).
\end{align*}
The last step follows from \cite[Ch. VII Ex. 2(c)]{Mac1995}. Here,
\[
    c_\lambda (X) = \prod_{(i,j) \in \lambda} (X-i +2j -1)
\]
is the \emph{content polynomial} (see \cite[Ch. VII (2.24)]{Mac1995}). For the partition $\mu = (1,\ldots,1)$, we find
\[
    c_{(1,\ldots,1)} (X) = \prod_{i=1}^n (X-i+1).
\]
Clearly, we have $c_{(1,\ldots,1)}(4) = 0$ for all $n \ge 5$.
Therefore, by \cite[Corollary 1.5]{BambergLansdown}, $S$ consists
of half of the perfect matchings. 
This is a contradiction, because there is an odd number of perfect matchings of $K_{2n}$.
\end{example}

Table 1 gives explicit constrains that follow from Theorem \ref{theorem:twoparts} and Corollary  \ref{cor:derivation} for $\lambda$-factorisations of index 1 where $\lambda$ has few and small parts.

\begin{center}
\begin{table}[!ht]\label{table:small constraints}
\begin{tabular}{b{9cm}m{4cm}}
\begin{tabular}{ll}
\toprule 
 $\lambda$  & Constraints\\
\midrule
($n-t$,1,$\ldots$,1) &   $t \neq n-2$  \\
($n-1$,1)  & none, always exists  \\
($n-2$,2)  & $n \equiv 0 \pmod 3$ \\
($n-3$,3)  & $n \equiv 0,10 \pmod{15}$\\
($n-3$,2,1)  & $n \equiv 1 \pmod 3$\\
($n-4$,4)  &  $n \equiv 0,21 \pmod{35}$\\
($n-4$,3,1)  &  $n \equiv 1,11 \pmod{15}$\\
($n-4$,2,2)  &  does not exist \\
($n-4$,2,1,1)  &  does not exist  \\
($n-5$,5)  &  $n \equiv 0,36,126,162,225,252 \pmod{315}$\\
\bottomrule
\end{tabular}
\end{tabular}
\vspace{1ex}
\caption{Criteria for the existence of $\lambda$-factorisations of index 1 for some $\lambda$.}
\end{table}
\end{center}

\begin{Remark}
A near-perfect matching is a matching in a graph where all but one vertex are matched, and the graph must have an odd number of vertices. We can define $\lambda$-factorisations for a complete graph on an odd number of vertices, say $2n-1$, by considering only the partitions $\lambda \vdash 2n-1$ that have exactly one odd part. Every such partition extends uniquely to a partition $\widehat{\lambda} \vdash 2n$ that only has even parts. Likewise, every near-perfect matching of $K_{2n-1}$ extends uniquely to a perfect matching of $K_{2n}$. Conversely, deleting the element $2n$ of a set partition of $\{1,\ldots,2n\}$ with only even parts (resp. a perfect matching) leaves a set partition with exactly one odd part (resp. a near-perfect matching). Thus, there is a bijection between near-perfect matchings of $K_{2n-1}$ and perfect matchings of $K_{2n}$ and the corresponding $\lambda$-factorisations are essentially the same structures.\par
We could use similar techniques from the theory of association
schemes to study near-perfect matchings of $K_{2n-1}$. Indeed,
we obtain a natural association scheme from the Gelfand pair $(S_{2n-1},B_{n-1})$ where
$B_{n-1}$ is the centraliser of $(1\,2)(3\,4)\cdots (2n-3\,2n-2)$.
See \cite[\S11]{Lindzey} for more.
\end{Remark}

\section{Examples for small $n$}\label{section:smalln}

We give an overview of the existence of $\lambda$-factorisations of index 1 for small $n$.
The first case is $n=3$, with the only nontrivial partition being $\lambda=(2,1)$, and
a $(n-1,1)$-factorisation of index 1 is a 1-factorisation (see Example \ref{example:lambda-factorisations}), 
and so exists. For $n=4$, we have $\lambda\in \{(2,1,1),(2,2),(3,1)\}$, where again,
the $\lambda=(3,1)$ case is not interesting. 
Corollary \ref{cor:k_and_l} shows that there is no example for $\lambda=(2,2)$. 
It was shown by Cameron \cite[Theorem 7.6]{Cam1976} and Mathon \cite{Mathon} that there is no example for $\lambda=(2,1,1)$. 

For $n=5$, the partitions are totally ordered in the dominance order:
$(1,1,1,1,1) \ltriangle (2,1,1,1) \ltriangle (2,2,1) \ltriangle (3,1,1) \ltriangle (3,2) \ltriangle (4,1)\ltriangle (5)$.
By Corollary \ref{cor:k_and_l}, a $\lambda$-factorisation of index 1 does not
exist for $\lambda\in\{(2,2,1),(3,2)\}$. By Corollary \ref{cor:derivation},
a $(2,1,1,1)$-factorisation of index 1 does not exist.
Again, a $(4,1)$-factorisation of index 1 is a
1-factorisation (and so exists), and a $(3,1,1)$-factorisation is a hyperfactorisation,
which exists in this case (see Example \ref{example:abstracthyperovals}).

Now suppose $n=6$. 
We can immediately rule out $\lambda$-factorisations of index 1 for $\lambda\in\{(2,2,1,1),(2,2,2),(3,2,1),(3,3)\}$ with a simple application of Corollary \ref{cor:k_and_l}. 
Corollary \ref{corollary:2111} rules out $\lambda=(2,1,1,1,1)$. 
So the only nontrivial partitions of 6 that survive 
are $(3,1,1,1)$, $(4,1,1)$, and $(4,2)$. A $(4,1,1)$-factorisation of index 1 is 
equivalent to an abstract hyperoval of order 10.
It was shown in \cite{Lam_et_al} that such an object does not exist,
via an exhaustive computer search. 
There is an example of a $(4,2)$-factorisation of index 1 -- which is explored in Example \ref{example:42}. 
To show that $(3,1,1,1)$-factorisations of index 1 do not exist, we use the fact (Theorem \ref{theorem:derivation}) that the derivation of a putative example would yield a hyperfactorisation of $K_{10}$, and we know there are only two examples (up to isomorphism). So we construct the two examples, extend them to sets of perfect matchings of $K_{12}$ (by adjoining $\{11,12\}$ to each matching), and then we use a constraint satisfaction solver such as \textsf{Gurobi} \cite{gurobi} to show that there is no $(3,1,1,1)$-factorisation whose derivation is the given hyperfactorisation. The GAP code is attached as an appendix, where we make use of the package \textsf{Gurobify} \cite{gurobify}.

\begin{center}
\begin{table}[H]
\begin{tabular}{b{9cm}m{4cm}}
\begin{tabular}{ll}
\toprule 
 $\lambda$  & Exists?\\
\midrule
(2,1,1,1,1) &   No, Corollary \ref{corollary:2111}  \\
(2,2,1,1)  & No, Corollary \ref{cor:k_and_l}  \\
(2,2,2)    & No, Corollary \ref{cor:k_and_l}  \\
(3,1,1,1)  & No (by computer) \\
(3,2,1)    & No, Corollary \ref{cor:k_and_l} \\
(3,3)      & No, Corollary \ref{cor:k_and_l}, Theorem \ref{theorem:twoparts} \\
(4,1,1)    &  No, \cite{Lam_et_al} \\
(4,2)      &  Yes    \\
(5,1)      &  Yes     \\
\bottomrule
\end{tabular}
&
\begin{tikzpicture}[node distance=5mm and 5mm, every node/.style={outer sep=2pt,inner sep=0pt},scale=0.45]

\node (6)              at (0,-1)     {\youngdiagram{6}};
\node (51)             at (0,-2.5) {\youngdiagram{5,1}};
\node (42)             at (0,-4)  {\youngdiagram{4,2}};
\node (411)            at (-3,-5)   {\youngdiagram{4,1,1}};
\node (33)             at (3,-5)    {\youngdiagram{3,3}};
\node (321)            at (0,-6)    {\youngdiagram{3,2,1}};
\node (3111)           at (-3,-7) {\youngdiagram{3,1,1,1}};
\node (222)            at (3,-7)  {\youngdiagram{2,2,2}};
\node (2211)           at (0,-8)  {\youngdiagram{2,2,1,1}};
\node (21111)          at (-1.5,-10)   {\youngdiagram{2,1,1,1,1}};
\node (111111)         at (1.5,-11)  {\youngdiagram{1,1,1,1,1,1}};

\draw (6) -- (51);
\draw (51) -- (42);
\draw (42) -- (33);
\draw (42) -- (411);

\draw (411) -- (321);
\draw (33) -- (321);
\draw (321) -- (3111);

\draw (321) -- (222);

\draw (3111) -- (2211);
\draw (222) -- (2211);
\draw (2211) -- (21111);

\draw (21111) -- (111111);

\end{tikzpicture}
\end{tabular}
\caption{The case $n=6$. On the right is the dominance Hasse diagram of the partitions of 6.
On the left is the existence data for $\lambda$-factorisations of index 1.}
\end{table}
\end{center}

\begin{example}\label{example:42}
Here we give an example of a $(4,2)$-factorisation.
First we let $\mathrm{AGL}(1,11)$ be the stabiliser of $\infty$ in the action
of $\mathrm{PGL}(2,11)$ on the projective line $\PG(1,11)$. 
Let $\F_{11}^{\square}$ be the set of nonzero
squares in the finite field $\F_{11}$ of order 11.
Consider the following two perfect matchings on the 12 points
of the projective line $\PG(1,11)=\F_{11}\cup \{\infty\}$:
\[
        \{\{0,\infty\}\}\cup\{\{x,-x\} : x\in\F_{11}^\square\},
        \quad \{\{0,\infty\}\}\cup\{\{x,7x\} : x\in\F_{11}^\square\}.
\]
    Then the orbits of $\mathrm{AGL}(1,11)$ on these perfect matchings are orbits of size 11 and 22, respectively.
    Their union forms a $(4,2)$-factorisation of index 1. Note that this example is very much related to
    the Steiner system $S(5,6,12)$ -- the \emph{small Witt design}. If we take the subset $\{\infty\}\cup \F_{11}^\square$ of the projective line $\PG(1,11)$ and its orbit under $\mathrm{PSL}(2,11)$, then we have 132 subsets of size 6 such that any 5 points lie in a unique block.

    The authors failed to generalise this example to more than this one example, despite the large amount of symmetry that this example possesses. It would be interesting to 
    find $\lambda$-factorisations admitting $\mathrm{AGL}(1,q)$ (as automorphisms),
    where $2n=q+1$.
\end{example}

For $n=7$, we can again apply Corollary \ref{cor:k_and_l} and Corollary \ref{corollary:2111} to show
that most of the $\lambda$-factorisations (of index 1) do not exist. 
The only nontrivial partitions to consider are $(4,2,1)$, and $(5,1,1)$.
Now a $(5,1,1)$-factorisation of index 1
is a hyperfactorisation of $K_{14}$ of index 1. To the authors' knowledge, it is unknown whether
such a hyperfactorisation exists. If it were to exist, then it would also give rise to a $5$-design
on 14 points with block size $6$ and index $15$ (see \cite[p. 10]{BorJunVan1991}). We emphasise that this $5$-design would have repeated blocks as the complete design on 14 points with block size 6 is a $5$-design of index~9. There is a $5$-design on 14 points with block size 6 and index~3 due to Brouwer~\cite{Brouwer} that does not have repeated blocks. The $5$-design obtained from a hypothetical $(5,1,1)$-factorisation of index 1 could be a 5-fold copy of the design by Brouwer. Finally, the authors could not determine
by computer or otherwise whether a $(4,2,1)$-factorisation of index 1 exists or not.

For $n=8$, after applying Corollary \ref{cor:k_and_l} and Corollary \ref{corollary:2111},
we are left with nontrivial $\lambda$-factorisations (of index 1) with $\lambda\in\{(5,1,1,1),(6,1,1)\}$.
Again, we do not know if these partitions have examples or not. One can see a pattern emerging here
where the possible valid partitions have relatively small length. Finally,
for $n=9$, we have an example of a hyperfactorisation of index 1 of $K_{18}$ from Cameron's construction,
and so the case $\lambda=(7,1,1)$ turns out to have an example. So the remaining partitions where we
do not have a resolution to the existence question are $(5,1,1,1,1), (6,1,1,1), (7,2)$.

\section{Open questions}

We list some open questions for further research that the reader might be interested to work on. We expect that they are roughly ordered increasing in difficulty.
\begin{enumerate}
    \item{
        Can we construct more $(n-2,2)$-factorisations that are not also hyperfactorisations like Example \ref{example:42}? Is there an infinite family?
    }
    \item{
        Can we construct any non-trivial $(n-3,3)$-factorisation? Since this would also yield an $(n-2,2)$-factorisation by Theorem \ref{theorem:comparison of strengths}, one should investigate $(n-2,2)$-factorisations first. 
    }
    \item{
        Can we construct any non-trivial $(n-3,1,1,1)$-factorisation?
    }
    \item{
        Can we investigate similar questions for the coset algebra $S_{kn} / \left(S_k \wr S_n\right)$ of uniform set partitions? Since this is not a Gelfand pair for $k \ge 3$, a different approach is needed.
    }
\end{enumerate}

\section*{Acknowledgements}

The second author was partially funded by the Deutsche Forschungsgemeinschaft (DFG, German Research Foundation) – Project-ID 491392403 – TRR 358. He would also like to thank the DAAD (German Academic Exchange Service) for funding a 2-month research visit to the University of Western Australia during which most of the research on this project took place.

Both authors are extremely grateful to Professor John R. Stembridge for answering a question of ours on zonal spherical functions and pointing us to the relevant information in the literature.
They are also grateful for correspondence they had with Sam Mattheus, Professor Arun Ram and Professor Ole Warnaar.

The authors would like to thank Professor Alfred Wassermann for directing us to the work of Brouwer \cite{Brouwer} referenced in Section~\ref{section:smalln}.

\bibliography{references}
\bibliographystyle{amsplain}

\section*{Appendix: GAP code}

\begin{lstlisting}
# First, we construct the two hyperfactorisations of K_{10}
n := 5;
s2n := SymmetricGroup(2*n);
matching := List([1..n], i -> [2*i-1,2*i]);
matchings := Set(Orbit(s2n, matching, OnSetsSets));;
perm := Action(s2n, matchings, OnSetsSets);;
maxs := MaximalSubgroupClassReps(perm);;
s9 := First(maxs, t -> Size(t)=Factorial(9));
IsSymmetricGroup(s9);
maxs := MaximalSubgroupClassReps(s9);;
m := First(maxs, t -> Size(t)=432);
subgroup216 := MaximalNormalSubgroups(m)[1];
orbits := Filtered(Orbits(subgroup216), t -> Size(t)<=63);;
Sort(orbits, {x,y} -> Size(x) > Size(y));
# Both examples below are equivalent to Mathon's example
o1 := Union(orbits{[1,3]});;
o2 := Union(orbits{[2,3]});;

psl := PSL(2,8);
iso := IsomorphismGroups(SymmetricGroup(9),s9);;
psl_on_matchings := Image(iso, psl);
# This example arises from a hyperoval of PG(2,8) ... 
# Cameron's example
o3 := First(Orbits(psl_on_matchings), t -> Size(t)=63);;

# Now extend each one with [11,12] to get a partial configuration
extensions := List([matchings{o1},matchings{o3}], mm -> 
	List(mm, t -> Concatenation(t,[[11,12]])));;

# now change n and set up the Gurobi program
LoadPackage("gurobify");
n := 6;
s2n := SymmetricGroup(2*n);
matching := List([1..n], i -> [2*i-1, 2*i]);
matchings := Set(Orbit(s2n, matching, OnSetsSets));;
perm := Action(s2n, matchings, OnSetsSets);;
lambda := [3,1,1,1];
# one set partition of shape 2*lambda
p := []; max := 0;
for s in 2*lambda do
	Add(p, max + [1..s]);
	max := Maximum(Union(p));
od;
refines := {x,y} -> ForAll(x, t -> ForAny(y, u -> IsSubset(u,t)));
cp := Filtered(matchings, t -> refines(t,p));;
antidesign1 := SubsetToCharacteristicVector(cp, matchings);;
# These are the antidesigns, as 0,1-vectors.
# We seek a set that intersects each antidesign in precisely
# 1 element.
antidesigns := Orbit(perm, antidesign1, Permuted);;

for i in [1,2] do
	Print("Doing example ", i, "\n");
	ext := extensions[i];
	varnames := List([1..Size(matchings)], i -> Concatenation("x", String(i)));;
	vartypes := ListWithIdenticalEntries( Size(varnames), "Binary" );;
	model := GurobiNewModel(vartypes, varnames);;
	
	# constraint: meets each antidesign in 1 element
	GurobiAddMultipleConstraints(model,antidesigns,
		ListWithIdenticalEntries(Size(antidesigns),"="),
		ListWithIdenticalEntries(Size(antidesigns),1));
	
	# constraint: derives to known hyperfactorisation
	GurobiAddConstraint(model, 
            SubsetToCharacteristicVector(ext,matchings), "=", 63);	
	status := GurobiOptimiseModel(model);
	
	if status = 3 then 
		Print("no example found\n");
	else
		Print("something else has happened\n");
	fi;
od;
\end{lstlisting}

\end{document}